\documentclass[11pt,reqno]{amsart}
\usepackage{amsthm, amsmath, amssymb, amscd, latexsym, multicol, verbatim, enumerate, }
\usepackage{hyperref}

\usepackage[top=1.15in, bottom=1.15in, left=1.17in, right=1.17in]{geometry}
\usepackage[english]{babel}
\usepackage{mathrsfs}
\usepackage[all]{xy}
\usepackage[dvips]{graphicx}

\newcounter{cnt}
 \makeatletter
\def\mydggeometry{\makeatletter\dg@YGRID=1\dg@XGRID=20\unitlength=0.003pt\makeatother}
\makeatother \theoremstyle{remark}
\numberwithin{equation}{section}

\theoremstyle{definition} 
\newtheorem{definition}{Definition}\theoremstyle{definition}
\newtheorem{proposition}{Proposition}
\newtheorem{theorem}{Theorem}
\newtheorem{lemma}{Lemma}
\newtheorem{corollary}{Corollary}

\newtheorem{example}{Example}

\newtheorem*{thma}{Theorem A}
\newtheorem*{thmb}{Theorem B}
\newtheorem*{thmc}{Theorem C}

\newcommand{\g}{\mathfrak{g}}
\newcommand{\n}{\mathfrak{n}}

\newcommand{\ra}{\rightarrow}
\newcommand{\ve}{\varepsilon}

\newcommand{\ts}{\otimes}
\newcommand{\s}{\sigma}
\newcommand{\ff}{\mathcal{F}}

\newcommand{\mc}{\mathbb{C}}
\newcommand{\mk}{\mathbb{K}}

\newcommand{\ad}{\operatorname*{ad}}
\newcommand{\ev}{\operatorname*{ev}}

\newcommand{\Hom}{\operatorname*{Hom}}

\newcommand{\Ext}{\operatorname*{Ext}}
\newcommand{\gr}{\operatorname*{gr}}
\newcommand{\GL}{\operatorname*{GL}}

\newcommand{\lie}{\mathfrak}
\newcommand{\rep}{\operatorname*{rep}}

\newcommand{\sspan}{\operatorname*{span}}
\newcommand{\id}{\operatorname*{id}}
\newcommand{\End}{\operatorname*{End}}
\newcommand{\diag}{\operatorname*{diag}}
\newcommand{\Mat}{\operatorname*{Mat}}

\begin{document}

\author{Xin Fang, Ghislain Fourier, Markus Reineke}
\address{Xin Fang: Mathematisches Institut, Universit\"{a}t zu K\"{o}ln, Weyertal 86-90, D-50931, K\"{o}ln, Germany.}
\email{xfang@math.uni-koeln.de}
\address{Ghislain Fourier: Mathematisches Institut, Universit\"at Bonn}
\address{School of Mathematics and Statistics, University of Glasgow}
\email{ghislain.fourier@glasgow.ac.uk}
\address{Markus Reineke: Fachbereich C -- Mathematik, Bergische Universit\"at Wuppertal}
\email{mreineke@uni-wuppertal.de}


\date{\today}

\title[Quantum PBW filtration]{PBW-type filtration on quantum groups of type $A_n$}
\begin{abstract}
We will introduce an $\mathbb{N}$-filtration on the negative part of a quantum group of type $A_n$, such that the associated graded algebra is a $q$-commutative polynomial algebra. This filtration is given in terms of the representation theory of quivers, by realizing the quantum group as the Hall algebra of a quiver. We show that the induced associated graded module of any simple finite-dimensional module (of type 1) is isomorphic to a quotient of this polynomial algebra by a monomial ideal, and we provide a monomial basis for this associated graded module.\\
This construction can be viewed as a quantum analog of the classical PBW framework, and in fact, by considering the classical limit, this basis is the monomial basis provided by Feigin, Littelmann and the second author in the classical setup.
\end{abstract}
\maketitle

\section*{Introduction}
\subsection*{PBW filtration - revisited}
Let $\g$ be a simple Lie algebra with Cartan decomposition $\g=\lie n^+\oplus \lie h\oplus \lie n^-$ and simple roots $\Pi=\{\alpha_1,\alpha_2,\cdots,\alpha_n\}$. 
By setting the degree of any non-zero element of $\lie n^-$ to $1$, and considering the induced filtration of $U(\lie n^-)$, it follows from the PBW theorem that the associated graded algebra is isomorphic to $S(\lie n^-)$, the polynomial algebra of the vector space $\lie n^-$. 
If one fixes for any simple, finite-dimensional module $V$ a highest weight vector, then this PBW filtration induces a filtration on $V$, such that the associated graded module is a  $\mathbb{N}$-graded $S(\lie n^-)$-module. \par
This framework of PBW filtrations has been introduced in \cite{FFoL1}, and various aspects of this construction have been studied in recent times. 
It gained a lot of attention due to its connection to different subjects such as degenerations of flag varieties \cite{F1}, toric degenerations of flag varieties \cite{FFoL3}, Newton-Okounkov bodies \cite{FFoL3}, poset polytopes \cite{ABS1, Fo1}, quiver Grassmannians \cite{CLFR1, CLFR2}, Schubert varieties \cite{BF1, CL1, CLL1, Fo2},  graded characters \cite{BBDF1, CF1}, non-symmetric Macdonald polynomials \cite{CF1, FM1}, to name but a few.
\par
The annihilating ideal of $V^a(\lambda)$ as a quotient of $S(\lie n^-)$ has been provided for $\lie g = \lie{sl}_{n+1}$ \cite{FFoL1}, $\lie g = \lie{sp}_{2n}$ \cite{FFoL2} and cominuscule weights in other types \cite{BD1}. Moreover, a monomial basis was provided, parametrized by lattice points of a normal polytope $P(\lambda)$.\par
We should point out that the annihilating ideal of $V^a(\lambda)$ is not monomial in general, so the basis is obtained by actually choosing a homogeneous total order on monomials in $S(\lie n^-)$. This order provides an $\mathbb{N}^N$-filtration ($N=\dim\lie n^-$) and an induced graded module $V^t(\lambda)$. This is again an $S(\lie n^-)$-module and by construction, its annihilating ideal is monomial.\par
Suppose that $\g=\mathfrak{sl}_{n+1}$. We provide a new grading of the root vectors in $\lie n^-$, hence of the generators of $S(\lie n^-)$, as follows:
\[ 
\operatorname{\deg} (f_{\alpha_{i} + \ldots + \alpha_j})  = (j-i+1)(n-j+1).
\]
This endows $U(\lie n^-)$ with the structure of an $\mathbb{N}$-filtered algebra, whose associated graded algebra is again isomorphic to $S(\lie n^-)$. We denote the induced associated graded module of the simple module $V(\lambda)$ by $V^{\ff}(\lambda)$.  The first main result of this paper is
\begin{thma}\label{main-classical} 
Let $\lie g = \lie{sl}_{n+1}$ and $\lambda$ a dominant integral weight, then:
\begin{enumerate}
\item The lattice points in $P(\lambda)$ parametrize also a monomial basis of $V^{\ff}(\lambda)$.
\item The annihilating ideal of $V^{\ff}(\lambda)$ is monomial.
\end{enumerate}
\end{thma}
So instead of considering a homogeneous total order, we adjust the $\mathbb{N}$-grading on root vectors, to obtain an $\mathbb{N}$-graded module with the same monomial basis, whose annihilating ideal is also monomial. So we are loosing less structure than in the totally ordered case but still gain the nice property of having a monomial ideal.\par
The first part in the theorem seems to be the same as the main theorem in \cite{FFoL1}, but they are essentially different: in the original PBW filtration (\emph{loc. cit.}), a particular total order is fixed in the very beginning to select elements in $S(\lambda)$ from $V(\lambda)$; when the new filtration is under consideration, an $\mathbb{N}$-filtration suffices.\par
Our main goal was a different one:  We wanted to introduce and study a PBW filtration and the corresponding degenerations for quantum groups. 

\subsection*{Quantum groups}
Analogous to the classical setup, a PBW filtration for quantum groups should be an $\mathbb{N}$-filtration of $U_q(\lie n^-)$ such that the associated graded algebra is isomorphic to $S_q(\lie n^-)$, the $q$-commutative polynomial algebra on the vector space $\lie n^-$. Here  $U_q(\g)$ is the Drinfeld-Jimbo quantum group associated to a simple Lie algebra $\g$ and $U_q(\lie n^-)$ is its negative part. For any reduced decomposition of $w_0$, the longest element in the Weyl group $W$ of $\g$, Lusztig \cite{Lusztig} constructed a family of quantum PBW root vectors $F_\beta\in U_q(\lie n^-)$ parametrized by $\Delta_+$, such that an analogue of the PBW theorem holds for $U_q(\lie n^-)$. \par
But if one tries to construct the PBW filtration,  requiring the quantum PBW root vectors $F_\beta$ to be of degree $1$, this fails to  give an $\mathbb{N}$-filtration on $U_q(\lie n^-)$ (counter-examples are provided in Section~\ref{counter-example}).
\par
Our first task is to seek for an $\mathbb{N}$-filtration on $U_q(\lie n^-)$ such that the associated graded algebra is $S_q(\lie n^-)$. When the Lie algebra $\lie n^-$ is of type $A$, $D$ or $E$, such a filtration can be derived using Hall algebras.

\subsection*{Filtration arising from Hall algebras}
A choice of an orientation of a simply-laced Dynkin diagram determines a quiver with its associated category of representations. The negative part of the corresponding quantized enveloping algebra can be realized as the Hall algebra of the quiver, which is defined in purely representation theoretic terms.

This Hall algebra approach to quantum groups defines a parametrization of the element $F_{[M]}$ of a PBW type basis of $U_q(\lie n^-)$ by isomorphism classes $[M]$ of finite dimensional representation of $Q$, and the multiplication of basis elements is described in terms of short exact sequences between representations.
An $\mathbb{N}$-valued degree function $\deg F_{[M]}$ can thus be viewed as a function on isomorphism classes of representations, also denoted by $\deg [M]$.

This opens up a way to study modified PBW filtrations and their defining degree functions in representation-theoretic terms. We give a representation-theoretic characterization of all degree functions such that the associated graded algebras are $q$-commutative polynomial algebras, which naturally leads to the canonical degree function studied in this paper:

\begin{thmb} 
A degree function on representations as above induces an $\mathbb{N}$-filtration on $U_q(\lie n^-)$ with associated graded algebra isomorphic to a $q$-commutative polynomial algebra $S_q(\lie n^-)$ if and only if it is of the form $\deg [M]=\dim{\rm Hom}(V,M)$, for a representation $V$ of $Q$ containing all non-projective and all simple projective indecomposable representations of $Q$ as a direct summand.
\end{thmb}

This theorem is proved using basic methods of Auslander-Reiten theory.

\subsection*{Application to \texorpdfstring{$\lie{sl}_{n+1}$}{the special linear algebra}}
We apply this construction to the case where $\g$ is the Lie algebra $\mathfrak{sl}_{n+1}$, and fix a reduced decomposition of $w_0$. For $F_{\alpha_i + \ldots  + \alpha_j} \in U_q(\lie n^-)$, a quantum PBW root vector corresponding to the fixed reduced decomposition of $w_0$, we define  $\text{deg}(F_{\alpha_i + \ldots  + \alpha_j})=(j-i+1)(n-j+1)$. 
The degree function considered here then canonically arises from the representation $V$ containing one copy of each indecomposable representation of $Q$ as a direct summand (see Definition~\ref{deg-fct}). 
Hence Theorem B applies and the algebra $U_q(\mathfrak{n}^-)$ becomes $\mathbb{N}$-filtered, such that the associated graded algebra is a $q$-commutative polynomial algebra $S_q(\lie n^-)$ on $\lie n^-$. \par
For a dominant weight $\lambda$, let $V_q(\lambda)$ be the finite dimensional (type 1) irreducible representation of $U_q(\g)$ of highest weight $\lambda$ and $\mathbf{v}_\lambda$ be a highest weight vector. 
The associated graded $S_q(\lie n^-)$-module will be denoted by $V_q^\ff(\lambda)$ with annihilating ideal $I_q^\ff(\lambda)\subset S_q(\lie n^-)$.
\par
We obtain a quantum version of Theorem A:

\begin{thmc} 
For a dominant weight $\lambda$,
\begin{enumerate}
\item $\{F^\mathbf{s} \mathbf{v}_\lambda|\ \mathbf{s}\in S(\lambda)\}$ forms a basis of $V_q^\ff(\lambda)$, where $S(\lambda)$ is the set of integral points in $P(\lambda)$, and $F^{\mathbf{s}}=F_{\beta_1}^{s_1}F_{\beta_2}^{s_2}\cdots F_{\beta_N}^{s_N}$ for $\mathbf{s}=(s_1,s_2,\cdots,s_N)\in S(\lambda)$;
\item $I_q^\ff(\lambda)$ is a monomial ideal.
\end{enumerate}
\end{thmc}
\par
The theorem follows from specializing the statements to the classical case and using Theorem A.
\bigskip 

The paper is organized as follows: In Section~\ref{one} we recall the PBW filtration, introduce the new grading and show that with respect to this new grading, the annihilating ideal is monomial. In Section~\ref{two}, we give the necessary definitions for quantum groups and explain the formerly known degree functions. In Section~\ref{three}, the new grading is motivated by considerations in Hall algebras, while the Theorem C is proved in Section~\ref{four}.


\section{PBW filtrations}\label{one}

A basic reference on Lie algebras is \cite{Hum}.

\subsection{Background}

Let $\g$ be the Lie algebra $\mathfrak{sl}_{n+1}$ of traceless matrices over $\mc$. Fix a Cartan decomposition $\g=\mathfrak{n}^+\oplus\mathfrak{h}\oplus\mathfrak{n}^-$ and a set of simple roots $\Pi=\{\alpha_1,\cdots,\alpha_n\}$ of $\g$. The positive roots of $\g$ are then of the form $\Delta_{+}=\{\alpha_{i,j}:=\alpha_i+\cdots+\alpha_j|\ 1\leq i\leq j\leq n\}$. Let $\varpi_i$, $i=1,\cdots,n$ be the fundamental weights, $\mathcal{P}$ be the weight lattice and $\mathcal{P}_+=\sum_{i=1}^n\mathbb{N}\varpi_i$ be the set of dominant weights. For a dominant integral weight $\lambda=m_1\varpi_1+m_2\varpi_2+\cdots+m_n\varpi_n\in\mathcal{P}_+$, let $V(\lambda)$ be the finite dimensional irreducible representation of $\g$ of highest weight $\lambda$. 
\par

\subsection{Original PBW filtration}\label{Sec:2.2}
We recall here the PBW filtration on $U(\lie n^-)$ and finite dimensional simple $\lie{sl}_{n+1}$-modules. For this, we define (following \cite{FFoL1}) a filtration
\[
U(\lie n^-)_s := {\sspan}_{\mathbb{C}}\{ x_{i_1} \cdots x_{i_\ell} \mid x_{i_j} \in \lie n^- \, , \, \ell \leq s\}.
\]
With this filtration, $U(\lie n^-)$ is a filtered algebra and by the PBW theorem, the associated graded algebra is $S(\lie n^-)$. We may also view $S(\lie n^-)$ as the universal enveloping algebra of $\lie n^{-,  a}$, the abelian Lie algebra on the vector space $\lie n^-$.
\par
Let $V(\lambda)$ be a simple $\lie{sl}_{n+1}$-module and $v_\lambda$ a highest weight vector. We have an induced filtration
\[
V(\lambda)_s := U(\lie n^{-})_s.v_\lambda.
\]
The associated graded module is by construction an $\lie n^{-, a}$-module denoted $V^a(\lambda)$. There is a bigger Lie algebra acting on $V^a(\lambda)$, namely $\lie g^a := \lie n^+ \oplus \lie h \oplus \lie n^{-,  a}$, where the action of $\lie n^+\oplus \lie h$ on $\lie n^{-, a}$ is obtained via the adjoint action of $\lie{sl}_{n+1}$.  With this we can state the following theorem, due to \cite{FFoL1}:
\begin{theorem}\label{ffl1} For $\lambda \in \mathcal{P}_+$, as an $S(\lie n^-)$-module,
\[V^a(\lambda) \cong S(\lie n^-)/I(\lambda),\]
where $I(\lambda)$ is the ideal generated by $\{ U(\lie n^+).f_\alpha^{(\lambda,\alpha)+1}|\ \alpha>0\} \subset S(\lie n^-)$.
\end{theorem}

It is important to notice here, that the annihilating ideal $I(\lambda)$ is not monomial. Let $\lambda = \varpi_2$ and $V(\varpi_2) = \bigwedge^2 \mathbb{C}^{n+1}$, then \[f_{\alpha_1 + \alpha_2} f_{\alpha_2 + \alpha_3} - f_{\alpha_2} f_{\alpha_1 + \alpha_2 + \alpha_3} \in I(\varpi_2)\] but none of the two monomials is.

\subsection{Total order and monomial ideal}\label{Sec:2.3}
Here, we recall the total order on monomials in $S(\lie n^-)$ introduced in \cite{FFoL1} and the definition of Dyck paths. This is an $\mathbb{N}^N$-order instead of an $\mathbb{N}$-order as for example the PBW partial order.
\par
Let the total ordering of the generators $f_{\alpha}$, $\alpha \in \Delta_+$, of $S(\lie n^-)$ be
\[
f_{\alpha_{n,n}}>f_{\alpha_{n-1,n}}>f_{\alpha_{n-2,n}}>\cdots>
f_{\alpha_{2,3}}>f_{\alpha_{2,2}}>f_{\alpha_{1,n}}>\cdots>f_{\alpha_{1,1}}.
\]

Then, using the reverse homogeneous lexicographic order, we obtain a total order on monomials in $S(\lie n^-)$. Since it is homogeneous, this total order is a refinement of the PBW partial order. We obtain an induced filtration on $V^a(\lambda)$ and denote the associated graded module $V^t(\lambda)$. This is an $S(\lie n^-)$-module and, by construction, the annihilating ideal is monomial.
\par
 There is an explicit description of a basis of this ideal or, equivalently, a monomial basis of $V^t(\lambda)$, due to \cite{FFoL3}. For this description, we recall the notion of Dyck paths:

\begin{definition} A (type $A$) Dyck path is a sequence of positive roots $\mathbf{p}=(\beta(0),\beta(1),\cdots,\beta(k))$ for $k\geq 0$ starting from a simple root $\beta(0)=\alpha_i=\alpha_{i,i}$ and ending at a simple root $\beta(k)=\alpha_j=\alpha_{j,j}$ for some $j\geq i$, such that if $\beta(s)=\alpha_{t,r}$, then $\beta(s+1)=\alpha_{t,r+1}$ or $\beta(s+1)=\alpha_{t+1,r}$. We denote the set of all Dyck path starting in $i$ and ending in $j$ by $\mathbb{D}_{i,j}$.
\end{definition}
With this we introduce a polytope $P(\lambda) \subset \mathbb{R}_{\geq 0}^{|\Delta_+|}$:
\[
P(\lambda) = \{ (s_\alpha)\in \mathbb{R}_{\geq 0}^{|\Delta_+|} \mid \sum_{\alpha \in \mathbf{p} } s_\alpha \leq m_i + \ldots + m_j \; , \; \forall \; \mathbf{p} \in \mathbb{D}_{i,j} \}.
\]
We denote the set of latttice points in $P(\lambda)$ by
\[
S(\lambda) := P(\lambda) \cap \mathbb{Z}^{|\Delta_+|},
\]
and for any $\mathbf{s} \in S(\lambda)$, we denote $f^{\mathbf{s}} := \prod_{\alpha > 0 } f_\alpha^{s_\alpha}$. 
The following theorem is due to \cite{FFoL1} and \cite{FFoL3}:
\begin{theorem}\label{ffl2}
The set $\{ f^{\mathbf{s}}v_\lambda \mid \mathbf{s} \in S(\lambda)\}$ is a basis of $V^t(\lambda)$. Furthermore, for any $\lambda, \mu \in \mathcal{P}_+$, their Minkowski sums satisfy:
\[
P(\lambda+ \mu) = P(\lambda) + P(\mu) \; \text{ and } S(\lambda + \mu) = S(\lambda) + S(\mu).
\]
\end{theorem}

Since we will need this later, we give an explicit description of the monomial basis for fundamental weights. Let $V = \bigwedge^k \mathbb{C}^{n+1}$ and $v = e_{i_1} \wedge \ldots \wedge e_{i_{k}} \in V$, where $\{ e_j \mid 1 \leq j \leq n+1\}$ is the standard basis of $\mathbb{C}^{n+1}$. Let $\{j_1 < \ldots < j_s \} = \{1, \ldots k \} \setminus \{i_1 , \ldots , i_k \}$ and $\ell_1 = i_k, \ell_{2} = i_{k-1}, \ldots, \ell_s = i_{k-s+1}$. Then we have in $V^t$:
\[v = f_{\alpha_{j_1} + \ldots + \alpha_{\ell_1-1}} \cdots f_{\alpha_{j_s} + \ldots + \alpha_{\ell_s -1}}.e_1 \wedge \cdots \wedge e_k.\]

\subsection{A new filtration on \texorpdfstring{$U(\mathfrak{n}^-)$}{the universal enveloping algebra}}
We introduce a new grading on the root vectors in $U(\lie n^-)$, consider the corresponding $\mathbb{N}$-filtration and show that, with this new grading, the set $S(\lambda)$ still parametrizes a basis of the associated graded module,  but with the annihilating ideal being monomial. \par
For this, we define for any $\alpha_{i,j} \in \Delta_+$:
$$\deg(f_{\alpha_{i,j}})=(j-i+1)(n-j+1).$$
By considering the total degree of a monomial, we extend this to $U(\lie n^-)$ and obtain a filtration:
$$\ff_n U=\sspan\{f=\prod_{\alpha>0} f_{\alpha}^{n_\alpha}|\ \deg(f)\leq n\}.$$

\begin{proposition}\label{Prop:Filt}
\begin{enumerate}
\item With this filtration, $U(\mathfrak{n}^-)$ is a filtered algebra.
\item The associated graded algebra $\gr_\ff U$ is isomorphic to $S(\lie n^-)$
\end{enumerate}
\end{proposition}

This proposition will be proved later in Section \ref{Sec:QGPBW}, after Proposition \ref{Prop:qFilt}.

\begin{example}
Let $\g=\mathfrak{sl}_4$. We list elements of several low degrees of $\gr_\ff U$:
$$
\begin{array}{cl}
\text{deg } 1: & f_{3,3}\\
\text{deg } 2: & f_{2,2}, f_{2,3}, f_{3,3}^2\\
\text{deg } 3: & f_{1,1}, f_{1,3}, f_{2,2}f_{3,3}, f_{2,3}f_{3,3}, f_{3,3}^3\\
\text{deg } 4: & f_{1,1}f_{3,3}, f_{1,2}, f_{1,3}f_{3,3}, f_{2,2}^2, f_{2,2}f_{2,3}, f_{2,3}^2, f_{2,2}f_{3,3}^2, f_{2,3}f_{3,3}^2, f_{3,3}^4
\end{array}
$$
\end{example}

Let $v_\lambda$ be a highest weight vector in $V(\lambda)$. The filtration $\ff_\bullet$ defined on $U(\mathfrak{n}^-)$ induces a filtration on $V(\lambda)$ by $\ff_sV(\lambda):=\ff_s U.v_\lambda$.
\par
Let $V^\ff(\lambda)$ denote the associated graded vector space. Then $V^\ff(\lambda)$ is an $S(\lie n^-)$-module. As $V(\lambda), V^a(\lambda), V^t(\lambda)$, we have $V^\ff(\lambda)$ generated by $v_\lambda$ as an $S(\lie n^-)$-module. We denote the annihilating ideal  $I^\ff(\lambda) \subset S(\lie n^-)$.

\begin{theorem}\label{main-thm}
For any dominant integral weight $\lambda\in\mathcal{P}_+$, the ideal $I^\ff(\lambda)$ is monomial. Furthermore, 
the set $\{f^{\mathbf{p}}v_\lambda|\ \mathbf{p}\in S(\lambda)\}$ forms a basis of $V^\ff(\lambda)$.
\end{theorem}

We will need several steps to prove the theorem. We start in the next section with proving the theorem for fundamental weights.

\subsection{The case of fundamental weights} 

\begin{proposition}\label{Prop:Fund}
For any $k=1,\cdots,n$, the set $\{f^{\mathbf{p}}v_{\varpi_k}|\ \mathbf{p}\in S(\varpi_k)\}$ forms a linear basis of $V^\ff(\varpi_k)$.
\end{proposition}

\begin{proof}
Let $\lambda=\varpi_k$ be a fundamental weight.
\par
We fix the standard basis $\{w_1,\cdots,w_{n+1}\}$ of the natural representation $V(\varpi_1)$, such that $f_{i,j}w_i=w_{j+1}$. Since $V(\varpi_k)=\Lambda^k V(\varpi_1)$, it has a basis $w_{i_1}\wedge\cdots\wedge w_{i_k}$ where $1\leq i_1<\cdots<i_k\leq n+1$. 
\par
The highest weight vector of $V(\varpi_k)$ is $v_\lambda:=w_1\wedge\cdots\wedge w_k$. We consider the vector $w_{i_1}\wedge\cdots\wedge w_{i_k}$ with $1\leq i_1<\cdots<i_k\leq n+1$. Let $1\leq s\leq k$ be the index satisfying 
$$1\leq i_1<\cdots<i_s\leq k<i_{s+1}<\cdots<i_k\leq n+1.$$
We denote $(j_1,\cdots,j_{k-s}):=\{1,\cdots,k\}\backslash\{i_1,\cdots,i_s\}$ where $1\leq j_1<\cdots<j_{k-s}\leq k$. 
\par
The monomials $f^{\mathbf{p}}$ such that $f^{\mathbf{p}}v_\lambda$ is proportional to $w_{i_1}\wedge\cdots\wedge w_{i_k}$ are of the form
$$f_{j_{\s(1)},i_k-1}f_{j_{\s(2)},i_{k-1}-1}\cdots f_{j_{\s(k-s)},i_{s+1}-1},$$
where $\s\in\mathfrak{S}_{k-s}$ is a permutation. The degree of such an element is given by:
\begin{eqnarray}\label{eq:monomial}
\sum_{t=1}^{k-s}(i_{k-t+1}-j_{\s(t)})(n+3-i_{k-t+1})=\text{const}-\sum_{t=1}^{k-s}j_{\s(t)}(n+3-i_{k-t+1}).
\end{eqnarray}
Since the sequence $\{i_k\}$ is increasing, the sequence $\{-n-3+i_{k-t+1}\}$ is decreasing with respect to $t$; it is clear that this degree attains its minimal value if and only if $j_{\s(1)}<j_{\s(2)}<\cdots<j_{\s(k-s)}$, i.e., $\s=\id$.
\par
Therefore, in $V^{\ff}(\varpi_k)$, we have
$$w_{i_1}\wedge\cdots\wedge w_{i_k}=f_{j_1,i_k-1}f_{j_2,i_{k-1}-1}\cdots f_{j_{k-s},i_{s+1}-1}.v_{\lambda}.$$
By definition, the monomial on the right hand side belongs to $S(\varpi_k)$. This proves that $\{f^{\mathbf{p}}v_{\varpi_k}|\ \mathbf{p}\in S(\varpi_k)\}$ generates $V^{\ff}(\varpi_k)$; these elements are clearly linearly independent by weight reasons, hence form a basis.
\end{proof}

The proof of the proposition looks formally similar to Theorem 8.8 in \cite{FFoL3}, but they are essentially different. 
In \cite{FFoL3}, a particular total order is chosen in the very beginning (see Example 8.1, \textit{loc.cit}), and the 
choice of the minimal element in the proof uses not only the PBW length grading but this auxiliary total order. 
In our construction, no particular total order is fixed, the choice relies only on the graded structure of $U(\mathfrak{n}^-)$.
\par
One of the advantages of this new grading on $U(\mathfrak{n}^-)$ is that it gives a \emph{canonical} choice of the compatible monomial basis of $V^\ff(\varpi_k)$.

\begin{proposition}\label{Prop:Mono}
For any $k=1,\cdots,n$, the annihilating ideal $I^\ff(\varpi_k)$ is monomial.
\end{proposition}

\begin{proof}
By the proof of Proposition~\ref{Prop:Fund}, it suffices to show that if $\mathbf{p} \notin S(\varpi_k)$, then $f^{\mathbf{p}} \in I^\ff(\varpi_k)$. But this follows straight from \eqref{eq:monomial}.
\end{proof}

We have seen before that the PBW filtration does not provide monomial ideals. So let us consider again $V = \bigwedge^2 \mathbb{C}^{n+1}$ but this time with respect to this new grading:
\begin{example}
Let $\g=\mathfrak{sl}_4$ and consider the representation $V(\varpi_2)=\bigwedge^2 \mathbb{C}^{4}$. The monomial ideal $I^\ff(\varpi_2)$ is generated by 
$$f_{1,1},\ \ f_{3,3},\ \ f_{1,3}^2,\ \ f_{1,2}^2,\ \ f_{2,3}^2,\ \ f_{2,2}^2,\ \ f_{22}f_{12},\ \ f_{22}f_{23},\ \ f_{12}f_{13},\ \ f_{23}f_{13},\ \ f_{12}f_{23}.$$
This ideal coincides with the annihilating ideal of the toric degenerated module $V^t(\varpi_2)$ in \cite{FFoL3}.
\end{example}

\subsection{Proof of the theorem}
The Minkowski sum property for  $S(\lambda)$ (Theorem~\ref{ffl2}) reduces the combinatorics of a general weight to the fundamental weights. Its proof is purely combinatorial and relies on the combinatorics of Dyck paths (more precisely, on the property of the corresponding marked chain polytope \cite{ABS1, Fo1}); moreover, it is independent of the grading on $U(\mathfrak{n}^-)$.\par
We turn to the proof of Theorem~\ref{main-thm}:
We use induction on the height of $\lambda$, $|\lambda|=\sum_{i=1}^n m_i$, the initial step being proved in Proposition~\ref{Prop:Fund} and \ref{Prop:Mono}. The induction step will be the proof for $\lambda + \mu$:
\begin{enumerate}
\item $\{f^{\mathbf{s}}v_{\lambda+\mu}|\ \mathbf{s}\in S(\lambda+\mu)\}$ forms a basis of $V^\ff(\lambda+\mu)$;
\item the annihilating ideal $I^\ff(\lambda+\mu)$ is monomial.
\end{enumerate}

In the following, let $<$ be a total order on $\Delta_+$, extending the partial order obtained via the degree function. Such a total order can be obtained by ordering all roots of the same degree linearly. 
\par
The following proposition is proved in \cite[Proposition 1.11]{FFoL3}. The statement there is for the associated graded module with respect to this refinement, but if a set is linear independent for the refinement graded space, then of course it is linear independent for $V^\ff(\lambda)$. 
\begin{proposition}\label{Prop:Indep}
For any $\lambda, \mu \in \mathcal{P}_+$, the set $\{f^{\mathbf{s}}(v_\lambda\ts v_{\mu})|\ \mathbf{s}\in S(\lambda)+S(\mu)\}$ is linearly independent in $V^\ff(\lambda)\ts V^\ff(\mu)$.
\end{proposition}

By defining 
$$\ff_n(V(\lambda)\ts V(\mu))=\sum_{k+l=n}\ff_kV(\lambda)\ts \ff_l V(\mu),$$
we have 
$${\gr}_{\ff}(V(\lambda)\ts V(\mu))=V^\ff(\lambda)\ts V^\ff(\mu).$$

Using Proposition~\ref{Prop:Indep} and Theorem~\ref{ffl2}, we see that 
$\{f^{\mathbf{s}}(v_\lambda\ts v_{\mu})|\ \mathbf{s}\in S(\lambda+\mu)\}$ is linearly independent in $V^\ff(\lambda)\ts V^\ff(\mu)$
and hence linearly independent in $V(\lambda)\ts V(\mu)$. This set is contained in the Cartan component $V(\lambda+\mu)$ of $V(\lambda)\ts V(\mu)$, and by dimension reasons, it is a basis of $V(\lambda+\mu)$. 

The following proposition proves (1).

\begin{proposition}\label{Prop:Proj}
The set $\{f^{\mathbf{s}}v_{\lambda+\mu}|\ \mathbf{s}\in S(\lambda+\mu)\}$ forms a basis of $V^\ff(\lambda+\mu)$. 
\end{proposition}

\begin{proof}
To show this, it suffices to prove that $\ff_nV(\lambda+\mu)\cap\mathcal{B}$ is a basis of $\ff_nV(\lambda+\mu)$. By definition,
$$\ff_nV(\lambda+\mu)\cap\mathcal{B}=\{f^{\mathbf{s}}v_{\lambda+\mu}|\ \deg(f^{\mathbf{s}})\leq n,\ \mathbf{s}\in S(\lambda+\mu)\}.$$
This set is linearly independent in $\ff_nV(\lambda+\mu)$, thus it suffices to show that it is generating. Take $\mathbf{s}\notin S(\lambda+\mu)$ such that $f^{\mathbf{s}}v_{\lambda+\mu}\in\ff_nV(\lambda+\mu)$ is not a basis element. We denote $\deg(f^{\mathbf{s}})=k+1\leq n$; since $\mathcal{B}$ is a basis of $V(\lambda+\mu)\subset V(\lambda)\ts V(\mu)$, we can write
\begin{equation}\label{Eq:1}
f^{\mathbf{s}}(v_{\lambda}\ts v_{\mu})=\sum_{\mathbf{t}\in S(\lambda+\mu)}c_{\mathbf{t}}f^{\mathbf{t}}(v_{\lambda}\ts v_{\mu}).
\end{equation}
\noindent
\textbf{Claim:}
For any $\mathbf{t}\in S(\lambda+\mu)$ such that $c_{\mathbf{t}}\neq 0$, we have $\deg(f^{\mathbf{t}})<\deg(f^{\mathbf{s}})$.
\\
\\
\noindent
\textit{Proof of the claim:}
Since $\mathbf{s}\notin S(\lambda+\mu)$, $f^{\mathbf{s}}(v_\lambda\ts v_{\mu})=0$ in $V^\ff(\lambda)\ts V^\ff(\mu)$. By the inductive hypothesis, in $V(\lambda)\ts V(\mu)$, we have
\begin{equation}\label{Eq:2}
f^{\mathbf{s}}(v_\lambda\ts v_{\mu})=\sum_{\mathbf{r}\in S(\lambda),\ \mathbf{p}\in S(\mu)}c_{\mathbf{r},\mathbf{p}}f^{\mathbf{r}}v_\lambda\ts f^{\mathbf{p}}v_{\mu}
\end{equation}
where $\deg(f^{\mathbf{r}})+\deg(f^{\mathbf{p}})\leq k$ provided $c_{\mathbf{r},\mathbf{p}}\neq 0$. 
\par
We take $\mathbf{t}\in S(\lambda+\mu)$ such that $f^{\mathbf{t}}$ is of maximal degree among $f^{\mathbf{r}}$ satisfying $c_{\mathbf{r}}\neq 0$. It suffices to show that $\deg(f^{\mathbf{t}})<\deg(f^{\mathbf{s}})$. By the Minkowski property, $S(\lambda+\mu)=S(\lambda)+S(\mu)$, hence there exist $\mathbf{t}_1\in S(\lambda)$ and $\mathbf{t}_2\in S(\mu)$ such that $\mathbf{t}=\mathbf{t}_1+\mathbf{t}_2$ and $cf^{\mathbf{t}_1}v_\lambda\ts f^{\mathbf{t}_2}v_{\mu}$ is a summand of $c_{\mathbf{t}} f^{\mathbf{t}}(v_\lambda\ts v_{\mu})$ with constant $c\neq 0$.
\par
We claim that for any $\mathbf{r}\in S(\lambda+\mu)$ such that $\mathbf{t}\neq\mathbf{r}$ and $\deg(f^{\mathbf{r}})\leq\deg(f^{\mathbf{t}})$, $f^{\mathbf{t}_1}v_\lambda\ts f^{\mathbf{t}_2}v_{\mu}$ is not a summand of $f^{\mathbf{r}}(v_\lambda\ts v_{\mu})$. Indeed, we write
$$f^{\mathbf{r}}(v_\lambda\ts v_{\mu})=\sum_{\mathbf{r}_1+\mathbf{r}_2=\mathbf{r}}d_{\mathbf{r}_1,\mathbf{r}_2}f^{\mathbf{r}_1}v_\lambda\ts f^{\mathbf{r}_2}v_{\mu}.$$
If $f^{\mathbf{t}_1}v_\lambda\ts f^{\mathbf{t}_2}v_{\mu}$ were a summand of $f^{\mathbf{r_1}}v_\lambda\ts f^{\mathbf{r}_2}v_{\mu}$ with $d_{\mathbf{r}_1,\mathbf{r}_2}\neq 0$, by the inductive hypothesis (if $f^{\mathbf{r}_1}v_\lambda=\sum_{\mathbf{p}\in S(\lambda)}f^{\mathbf{p}}v_\lambda$ and $f^{\mathbf{r}_2}v_{\mu}=\sum_{\mathbf{q}\in S(\mu)}f^{\mathbf{q}}v_{\mu}$ then $\deg(f^{\mathbf{r}_1})\geq\deg(f^{\mathbf{p}})$ and $\deg(f^{\mathbf{r}_2})\geq\deg(f^{\mathbf{q}})$), 
$$\deg(f^{\mathbf{r}})=\deg(f^{\mathbf{r}_1})+\deg(f^{\mathbf{r}_2})>\deg(f^{\mathbf{t}_1})+\deg(f^{\mathbf{t}_2})=\deg(f^{\mathbf{t}}),$$
where the strict inequality arises from the fact that $(\mathbf{t}_1,\mathbf{t}_2)\neq (\mathbf{r}_1,\mathbf{r}_2)$. This contradiction proves the claim.
\par
As a consequence, in formula (\ref{Eq:1}), the summand $f^{\mathbf{t}_1}v_\lambda\ts f^{\mathbf{t}_2}v_{\mu}$ has a non-zero coefficient and $\mathbf{t}_1\in S(\lambda)$, $\mathbf{t}_2\in S(\mu)$. By formula (\ref{Eq:2}), $\deg(f^{\mathbf{t}_1})+\deg(f^{\mathbf{t}_2})=\deg(f^{\mathbf{t}})\leq k<k+1=\deg(f^{\mathbf{s}})$.\par
This finishes the proof of the claim and hence the proof of the proposition.
\end{proof}

We are left with proving the monomiality of the annihilating ideal. This follows immediately from the following lemma 
\begin{lemma}\label{lem:cartan} For $\lambda, \mu \in \mathcal{P}_+$, the annihilating ideal of $V^\ff (\lambda)\odot V^\ff(\mu)$ is monomial and there exists an $S(\mathfrak{n}^-)$-module isomorphism
\[
V^\ff (\lambda)\odot V^\ff(\mu)\longrightarrow V^\ff(\lambda+\mu).
\]
\end{lemma}
\begin{proof}
Let $\mathbf{s} \notin S(\lambda) + S(\mu) = S(\lambda + \mu)$, then 
\[
f^{\mathbf{s}}(v_\lambda \otimes v_{\mu}) = \sum_{\mathbf{s_1} + \mathbf{s_2} = \mathbf{s} } c_{\mathbf{s}_1, \mathbf{s}_2} f^{\mathbf{s}_1}v_\lambda \otimes f^{\mathbf{s}_2}v_{\mu}
\]
where either $\mathbf{s}_1 \notin S(\lambda)$ or $\mathbf{s}_2 \notin S(\mu)$. Since by induction the annihilating ideals for the factors are monomial we have $f^{\mathbf{s}}(v_\lambda \otimes v_{\mu})  = 0$. And so the annihilating ideal of the Cartan component is monomial. \par
It now suffices to show that if $\mathbf{s} \notin S(\lambda + \mu)$, then $f^{\mathbf{s}}v_{\lambda + \mu} = 0$ in $V^\ff(\lambda + \mu)$. But this follows from Proposition~\ref{Prop:Proj} and the claim therein. This implies that there is a surjective map of $S(\lie n^-)$-modules 
\[
V^\ff (\lambda) \odot V^\ff(\mu)\longrightarrow V^\ff(\lambda+\mu)
\]
which is an isomorphism for dimension reasons. This proves Lemma~\ref{lem:cartan} and hence also the monomiality statement of Theorem~\ref{main-thm}.
\end{proof}

\subsection{Other types}\label{counter-example}
The monomiality of the annihilating ideal $I^\ff(\lambda)$ does not hold in general. Consider the Lie algebra $\mathfrak{so}_8$ of type $D_4$: let $Q$ be the quiver of type $D_4$ with node $2$ in the center as a sink. A grading arising from this quiver is: $\text{degree } 1: f_{4}$; $\text{degree } 2:  f_{12}, f_{23}, f_{24}$; $\text{degree } 5:  f_{123}, f_{124}, f_{234}$; $\text{degree } 6:  f_{1}, f_{2}, f_{3}, f_{12234}$; $\text{degree } 10:  f_{1234}$.
\par
Consider the weight space of weight $-\varpi_4$ in $V(\varpi_1+\varpi_3)$, it is of dimension 3. It is known that $f_1f_{12234}f_{234}+f_{124}f_{1234}f_{23}$ is in $I^\ff(\varpi_1+\varpi_3)$, but $f_1f_{12234}f_{234}$ and $f_{124}f_{1234}f_{23}$ could not be both in it. Similar counterexamples exist for any orientation of the quiver and any grading arising from this orientation.


\section{Quantum groups}\label{two}
A basic reference of quantum groups is \cite{Jantzen}.

\subsection{Recollections on quantum groups}\label{sec-2-1}
Let $\g$ be a simple Lie algebra of rank $n$ with Cartan matrix $C=(c_{ij})\in\Mat_n(\mathbb{Z})$. Let $D=\diag(d_1,\cdots,d_n)\in\Mat_n(\mathbb{Z})$ be a diagonal matrix symmetrizing $C$, thus $A=DC=(a_{ij})\in\Mat_n(\mathbb{Z})$ is the symmetrized Cartan matrix. Let $U_q(\g)$ be the corresponding quantum group over $\mc(q)$: it is a Hopf algebra generated by $E_i$, $F_i$ and $K_i^{\pm 1}$ for $i=1,\cdots,n$, subject to the following relations: for $i,j=1,\cdots,n$,
$$K_iK_i^{-1}=K_i^{-1}K_i=1,\ \ K_iE_jK_i^{-1}=q_i^{c_{ij}}E_j,\ \ K_iF_jK_i^{-1}=q_i^{-c_{ij}}F_j,$$
$$E_iF_j-F_jE_i=\delta_{ij}\frac{K_i-K_i^{-1}}{q_i-q_i^{-1}},$$
and for $i\neq j$,
$$\sum_{r=0}^{1-c_{ij}}(-1)^r E_i^{(1-c_{ij}-r)}E_jE_i^{(r)}=0,\ \ \sum_{r=0}^{1-c_{ij}}(-1)^rF_i^{(1-c_{ij}-r)}F_jF_i^{(r)}=0,$$
where 
\[
q_i=q^{d_i},\, \, [n]_q!=\prod_{i=1}^n \frac{q^n-q^{-n}}{q-q^{-1}},\  \ E_i^{(n)}=\frac{E_i^n}{[n]_{q_i}!}\ \ \text{and}\ \ F_i^{(n)}=\frac{F_i^n}{[n]_{q_i}!}.
\]

There exists a unique Hopf algebra structure $(\Delta,\ve,S)$ on $U_q(\g)$ such that, for $i=1,\cdots,n$, 
$$\Delta(K_i^{\pm 1})=K_i^{\pm 1}\ts K_i^{\pm 1},\ \ \Delta(E_i)=E_i\ts 1+ K_i\ts E_i,\ \ \Delta(F_i)=F_i\ts K_i^{-1}+1\ts F_i,$$
$$\ve(K_i^{\pm 1})=1,\ \ \ve(E_i)=0,\ \ \ve(F_i)=0,$$
$$S(E_i)=-K_i^{-1}E_i,\ \ S(F_i)=-F_iK_i,\ \ S(K_i)=K_i^{-1},\ \ S(K_i^{-1})=K_i.$$
Let $U_q(\mathfrak{n}^+)$ (resp. $U_q(\mathfrak{n}^-)$; $U_q^0$) be the sub-algebra of $U_q(\g)$ generated by $E_i$ (resp. $F_i$; $K_i^{\pm 1}$) for $i=1,\cdots,n$. There exists a triangular decomposition 
$$U_q(\g)\cong U_q(\mathfrak{n}^+)\ts U_q^0\ts U_q(\mathfrak{n}^-).$$

We fix several Lie theoretical notations. Let $W$ be the Weyl group of $\g$ with generators $s_1,\cdots,s_n$ and longest element $w_0\in W$. We fix a reduced expression $w_0=s_{i_1}\cdots s_{i_N}$ where $N=\#\Delta_+$. For $1\leq t\leq N$, we denote $\beta_t=s_{i_1}\cdots s_{i_{t-1}}(\alpha_{i_t})$; then $\Delta_+=\{\beta_t|\ t=1,\cdots,N\}$. The choice of the reduced expression of $w_0$ endows $\Delta_+$ with a total order such that $\beta_1<\beta_2<\cdots<\beta_N$.
\par
Let $T_i=T_{i,1}''$, $i=1,\cdots,n$ be Lusztig's automorphisms:
$$T_i(E_i)=-F_iK_i,\ \ T_i(F_i)=-K_i^{-1}E_i,\ \ T_i(K_j)=K_jK_i^{-c_{ij}},$$
for $i=1,\cdots,n$, and for $j\neq i$,
$$T_i(E_j)=\sum_{r+s=-c_{ij}}(-1)^rq_i^{-r}E_i^{(s)}E_jE_i^{(r)},\ \ T_i(F_j)=\sum_{r+s=-c_{ij}}(-1)^rq_i^{r}F_i^{(r)}F_jF_i^{(s)}.$$
We refer to Chapter 37 in \cite{Lusztig} for details. The PBW root vector $F_{\beta_t}$ associated to a positive root $\beta_t$ is defined by:
$$F_{\beta_t}=T_{i_1}T_{i_2}\cdots T_{i_{t-1}}(F_{i_t})\in U_q(\n^-).$$
The PBW theorem of quantum groups affirms that the set
$$\{F_{\beta_1}^{c_1}F_{\beta_2}^{c_2}\cdots F_{\beta_N}^{c_N}|\ (c_1,\cdots,c_N)\in\mathbb{N}^N\}$$
forms a $\mc(q)$-basis of $U_q(\mathfrak{n}^-)$ (\cite{Lusztig}, Corollary 40.2.2).
\par
For $\lambda\in\mathcal{P}_+$, we denote by $V_q(\lambda)$ the irreducible representation of $U_q(\g)$ of highest weight $\lambda$ and type $1$ with highest weight vector $\mathbf{v}_\lambda$.

\subsection{Specialization of quantum groups}
Details on the specialization of quantum groups can be found in \cite{HK02}, Section 3 and 4 of Chapter 3.
\par
Let $\mathbf{A}_1$ be the set of functions in $\mc(q)$ which are regular at $q=1$. It is a local ring with the maximal ideal $(q-1)$. The $\mathbf{A}_1$-form of $U_q(\g)$, denoted by ${}_{\mathbf{A}_1}U_q(\g)$, is defined as its $\mathbf{A}_1$-subalgebra generated by $E_i^{(k)}$, $F_i^{(k)}$, $K_i^{\pm 1}$, $(K_i;0)_q:=\frac{K_i-1}{q-1}$ for $i=1,\cdots,n$ and $k\geq 1$.
\par
The specialization of the quantum group $U_q(\g)$ at $q=1$ is the map 
$${}_{\mathbf{A}_1}U_q(\g)\ts_{\mathbf{A}_1}\mathbf{A}_1\ra {}_{\mathbf{A}_1}U_q(\g)\ts_{\mathbf{A}_1}\mc$$
by taking the tensor product with ${\ev}_1:\mathbf{A}_1\ra \mathbf{A}_1/(q-1)\cong\mc$, the evaluation map at $1$. The specialized algebra ${}_{\mathbf{A}_1}U_q(\g)\ts_{\mathbf{A}_1}\mc$ is isomorphic to the universal enveloping algebra $U(\g)$. Let $\xi:{}_{\mathbf{A}_1}U_q(\g)\ra U(\g)$ be the specialization map. We let $e_i$, $f_i$ and $h_i$, $i=1,\cdots,n$ denote the generators of $U(\g)$, then the specialization map $\xi$ is a Hopf algebra isomorphism sending $E_i\mapsto e_i$, $F_i\mapsto f_i$, $K_i^{\pm 1}\mapsto 1$ and $(K_i;0)_q\mapsto h_i$ for $i=1,\cdots,n$. Here $U(\g)$ is endowed with its standard Hopf algebra structure such that the space of primitive elements in $U(\g)$ is exactly $\g$.
\par
We turn to representations: the $\mathbf{A}_1$-form of the representation $V_q(\lambda)$ of $U_q(\g)$ is defined by: ${}_{\mathbf{A}_1}V_q(\lambda):={}_{\mathbf{A}_1}U_q(\g).\mathbf{v}_\lambda$. The specialization map is similarly defined as above:
$$\xi_\lambda:=\id\ts{\ev}_1:{}_{\mathbf{A}_1}V_q(\lambda)\ts_{\mathbf{A}_1}\mathbf{A}_1\ra {}_{\mathbf{A}_1}V_q(\lambda)\ts_{\mathbf{A}_1}\mc\cong V(\lambda).$$

\subsection{Specializations of PBW root vectors}
The aim of this subsection is to prove the following 

\begin{proposition}\label{Prop:SpePBW}
For any $\beta\in\Delta_+$, $\xi(F_\beta)\in\g\subset U(\g)$ is a non-zero element in the root space $\g_{-\beta}$.
\end{proposition}

\begin{proof}
Since $\xi$ is an isomorphism of algebra, $\xi(F_\beta)\in U(\g)_{-\beta}$ is a non-zero element. We divide the proof of $\xi(F_\beta)\in\g$ into two steps:
\begin{enumerate}
\item For any $i,j=1,\cdots,n$, we have $\xi(T_i(F_j))\in\g$. Indeed, for $i=j$, $T_i(F_i)=-K_i^{-1}E_i$, and thus, by definition, $\xi(T_i(F_i))=-e_i\in \g$. For $i\neq j$, by definition, 
$$\xi(T_i(F_j))=\sum_{r+s=-c_{ij}}(-1)^r f_i^{(r)}f_jf_i^{(s)},\ \ \text{where}\ \ f_i^{(r)}=\frac{f_i^r}{r!}.$$
This shows that $\xi(T_i(F_j))$ is proportional to $\ad(f_i)^{-c_{ij}}(f_j)\in\g\subset U(\g)$.
\item If $x\in U_q(\mathfrak{n}^-)$ such that $\xi(x)\in\g$, then for any $i=1,\cdots,n$, $\xi(T_i(x))\in\g$. Indeed, since $x\in U_q(\mathfrak{n}^-)$ and $\xi(x)\in\g$, the element $\xi(x)$ can be written as a sum of successive brackets of the $f_i$ for $i=1,\cdots,n$. Without loss of generality, by applying the Jacobi identity, we denote one of such brackets by $[f_{i_1},[f_{i_2},\cdots,[f_{i_{p-1}},f_{i_p}]\cdots]]$. Since $\xi$ is an algebra morphism, 
$$\xi(T_i(x))=\sum_{l=1}^p [f_{i_1},[\cdots,[\xi(T_i(F_{i_l})),[\cdots,[f_{i_{p-1}},f_{i_p}]\cdots]\cdots]].$$
By (1), $\xi(T_i(F_{i_l}))\in\g$ for any $l=1,\cdots,p$, proving $\xi(T_i(x))\in\g$.
\end{enumerate}
\end{proof}

\subsection{Why a quantum PBW-type filtration?}
The main motivation of the study of the quantum PBW filtration is to generalize the results in the classical case to quantum groups. We would like to mimic the strategy $V(\lambda)\leadsto V^a(\lambda)\leadsto V^t(\lambda)$ outlined in Section \ref{Sec:2.2} and \ref{Sec:2.3} to study the degenerations of an irreducible module $V_q(\lambda)$.
\par
So if we fix an order of the positive roots, we define the $q$-commutative polynomial algebra $S_q(\lie n^-)$ as the algebra generated by the homogeneous generators $\{F_{\beta_i}|\ i=1,\cdots,N\}$ and relations:
$$F_{\beta_j}F_{\beta_i}=q^{(\beta_i,\beta_j)}F_{\beta_i}F_{\beta_j},\ \ \ \text{for}\ \beta_i<\beta_j.$$
\par
The problem occurs in finding an object in the middle as an analogue of $V^a(\lambda)$. $U_q(\mathfrak{n}^-)$ can be generated by the PBW root vectors $F_{\beta_i}$, $i=1,\cdots,N$. We may copy the classical case to impose $\deg(F_{\beta_i})=1$. 

\begin{example}\label{Ex:counterex}
\begin{enumerate}
\item Let $\g=\mathfrak{sl}_4$ be of type $A_3$. Fix the reduced expression $w_0=s_1s_2s_1s_3s_2s_1$ of the longest element $w_0$ in the Weyl group of $\g$. We denote by $F_{i\ i+1\cdots k}$ the PBW root vector corresponding to the root $\alpha_i+\alpha_{i+1}+\cdots+\alpha_k$. The following relation holds in $U_q(\mathfrak{n}^-)$:
$$F_{23}F_{12}=F_{12}F_{23}-(q-q^{-1})F_2F_{123}.$$
\item Let $\g$ be of type $G_2$. Fix the reduced expression $w_0=s_1s_2s_1s_2s_1s_2$ of the longest element $w_0$ in the Weyl group of $\g$. We have in $U_q(\mathfrak{n}^-)$:
$$F_{3\alpha_1+2\alpha_2}F_{3\alpha_1+\alpha_2}=q^{-3}F_{3\alpha_1+\alpha_2}F_{3\alpha_1+2\alpha_2}+(1-q^{-2}-q^{-4}+q^{-6})F_{2\alpha_1+\alpha_2}^{(3)},$$
which specializes to $f_{3\alpha_1+2\alpha_2}f_{3\alpha_1+\alpha_2}=f_{3\alpha_1+\alpha_2}f_{3\alpha_1+2\alpha_2}$ in $U(\mathfrak{n}^-)$.
\end{enumerate}
\end{example}
According to this example, either the grading defined above does not give a filtered algebra, or the associated graded algebra is no longer $S_q(\mathfrak{n}^-)$. 
\par
We should mention here the de Concini-Kac filtration (\cite{DK}) which arises from a total order (lexicographic induced from a total order on the positive roots) on monomials in $U_q(\lie n^-)$. This filtration is not homogeneous, so this won't be a suitable filtration to mimic $V^a(\lambda)\leadsto V^t(\lambda)$.
\par
In brief, the known filtrations are not suitable for our purpose: this suggests us to search for new $\mathbb{N}$-filtrations on $U_q(\mathfrak{n}^-)$ by imposing different degrees on the PBW root vectors such that the associated graded algebra is a $q$-commutative polynomial algebra. In the next section, we will construct candidates using Hall algebras.


\section{Filtration arising from Hall algebras}\label{Sec:Hall}\label{three}
A basic reference on quiver representations and Auslander-Reiten theory is \cite{ASS}. For Hall algebras, we refer to \cite{Rin90}.

\subsection{Quiver representations}
Let $Q$ be a Dynkin quiver of type $A$, $D$ or $E$, that is, $Q$ has set of vertices $I$, and the number of arrows between two different vertices $i$ and $j$ (in either direction) equals $-a_{ij}$.
\par
Let $\mathbb{K}$ be an arbitrary field. We consider the category ${\rep}_{\mk}Q$ of finite dimensional $\mk$-representations of $Q$, which is an abelian $\mk$-linear category of global dimension at most one. We denote by $\mathbf{dim}M$ the dimension vector of a representation $M$, viewed as an element of the root lattice via 
$$\mathbf{dim}M = \sum_{i\in I} \dim(M_i)\alpha_i.$$ 
The homological Euler form 
$$\langle\mathbf{dim}M,\mathbf{dim}N\rangle = \dim\Hom(M,N)-\dim{\Ext}^1(M,N)$$ 
defines a (non-symmetric) bilinear form on the root lattice, whose symmetrization is the bilinear form defined by the Cartan matrix.
\par
The simple representations $S_i$ are naturally parametrized by the vertices $i\in I$ (that is, by the set of simple roots), and the indecomposable representations are naturally parametrized by the positive roots: for each $\alpha\in\Delta_+$, there exists a unique (up to isomorphism) indecomposable representation $U_\alpha$ such that $\mathbf{dim}U_\alpha = \alpha$. We can thus index the isomorphism classes of representations of $Q$ naturally by functions
from $\Delta_+$ to nonnegative integers. For such a function $\mathbf{m}: \Delta_+\ra \mathbb{N}$, we denote by $M(\mathbf{m})$ the $\mk$-representation $\bigoplus_{\alpha\in\Delta_+} U_\alpha^{\mathbf{m}(\alpha)}$. Note that this parametrization holds over arbitrary fields $\mk$.
\par
The category ${\rep}_{\mk}Q$ is representation-directed, which means that there exists an enumeration $\beta_1,\cdots,\beta_N$ of the positive roots such that $$\Hom(U_{\beta_k},U_{\beta_l})=0\mbox{ for }k>l\mbox{ and }{\Ext}^1(U_{\beta_k},U_{\beta_l})=0\mbox{ for }k\leq l.$$ It is known that there exists a sequence $i_1,\cdots,i_N$ in $I$ such that $s_{i_1}\cdots s_{i_N}$ is a reduced expression for the longest Weyl group element $w_0\in W$, and such that $\beta_k = s_{i_1}\cdots s_{i_{k-1}}(\alpha_{i_k})$ for all $k = 1,\cdots,N$.
\par
For a dimension vector $\mathbf{d} = \sum_{i\in I} d_i\alpha_i$, we define the variety of representations
$$R_{\mathbf{d}}(Q) = \bigoplus_{\alpha:i\ra j} \Hom(\mk^{d_i} , \mk^{d_j} ),$$ on which $G_{\mathbf{d}} = \prod_{i\in I} \GL(\mk^{d_i})$ acts via base change
$$(g_i)_i \cdot (f_\alpha)_\alpha = (g_if_\alpha g_i^{-1})_{\alpha:i\ra j}.$$
Tautologically, the $G_{\mathbf{d}}$-orbits $\mathcal{O}_M$ in $R_{\mathbf{d}}(Q)$ correspond bijectively to the isomorphism classes $[M]$ of $\mk$-representations of $Q$ of dimension vector ${\mathbf{d}}$. We call $N$ a degeneration of $M$, written $M \leq N$, if $\mathcal{O}_N$ belongs to the closure $\overline{\mathcal{O}_M}$. 
For Dynkin quivers, we have $M \leq N$ if and only if $\dim\Hom(V,M) \leq \dim\Hom(V,N)$ for all representations $V$ (equivalently, for all indecomposable representations $V$) of $Q$. In particular, $M$ and $N$ are isomorphic if and only if $\dim\Hom(V,M) = \dim\Hom(V,N)$ for all representations $V$ (equivalently, for all indecomposable representations $V$) of $Q$.

\subsection{Hall algebras} 
Let $\mk=\mathbb{F}_q$ be a finite field.
\par
Given three $\mk$-representations $M$, $N$ and $X$ of $Q$ with associated functions $\mathbf{m}, \mathbf{n}, \mathbf{x}:\Delta_+\ra \mathbb{N}$ as above, let $F_{M,N}^X$ be the number of sub-representations $U$ of $X$ which are isomorphic to $N$, with quotient $X/U$ isomorphic to $M$. These numbers are known to behave polynomially in the size of the field, that is, there exists a polynomial $F_{\mathbf{m},\mathbf{n}}^{\mathbf{x}}(u)\in \mathbb{Z}[u]$ such that its evaluation at the cardinality of the finite field $\mk$ equals to the number $F_{M(\mathbf{m}),M(\mathbf{n})}^{M(\mathbf{x})}$ just defined (we use the fact that we can model the representations over arbitrary fields). By abuse of notation, we also write $F_{M,N}^X(q)$ for $F_{\mathbf{m},\mathbf{n}}^{\mathbf{x}}(q)$.
\par
Define the Hall algebra $H(Q)$ as the $\mc(q)$-algebra with linear basis $u_{[M]}$ indexed by the isomorphism classes $[M]$ in $\rep_{\mk}(Q)$, and with multiplication 
$$u_{[M]} u_{[N]} =q^{\langle \mathbf{dim}M,\mathbf{dim}N\rangle} \sum_{[X]}F_{M,N}^X (q^2)u_{[X]}.$$ 
There exists a unique algebra isomorphism $\eta: U_q(\mathfrak{n}^-) \ra H(Q)$ which maps the Chevalley generator $F_i$ to $u_{[S_i]}$. It is compatible with the PBW-type basis of the quantum group in the following sense: we have 
$$\eta(F^{\mathbf{m}}) = F_{[M]}:= q^{\dim \End(M)-\dim M} u_{[M]}.$$
When $M=\bigoplus_{\alpha\in\Delta_+}U_\alpha^{\mathbf{m}(\alpha)}$, we have
$$F_{[M]}=F_{[U_{\beta_1}]}^{(\mathbf{m}(\beta_1))}\cdots F_{[U_{\beta_N}]}^{(\mathbf{m}(\beta_N))}.$$

\begin{lemma}\label{Lem:LSHall}
For $1\leq k<l\leq N$, we have
$$F_{[U_{\beta_l}]}F_{[U_{\beta_k}]}=q^{(\beta_k,\beta_l)}F_{[U_{\beta_k}]}F_{[U_{\beta_l}]}+\sum_{\mathbf{m}}c_{\mathbf{m}}F_{[U_{\beta_1}]}^{(\mathbf{m}(\beta_1))}\cdots F_{[U_{\beta_N}]}^{(\mathbf{m}(\beta_N))}$$
where all functions $\mathbf{m}$ appearing on the right hand side with non-zero coefficient are only supported on $\beta_{k+1},\cdots,\beta_{l-1}$.
\end{lemma}

\subsection{Filtrations induced by dimensions of Hom-spaces}
Let $R$ be a commutative ring, and let $A$ be a unital $R$-algebra which is free as an $R$-module. We consider ($\mathbb{N}$-)filtrations $\ff_\bullet = (\ff_0 \subset\ff_1 \subset \ff_2 \subset \cdots)$ on $A$ such that $\bigcup_n \ff_n = A$ and $\ff_m\ff_n \subset \ff_{m+n}$ for all $m,n \in \mathbb{N}$. We call $\ff_\bullet$ \emph{normalized} if $\ff_0=R.1$.
Let $\mathcal{B}$ be a basis of $A$ as an $R$-module. The basis $\mathcal{B}$ is called \emph{compatible with the filtration} $\ff_\bullet$ if $\ff_n \cap \mathcal{B}$ is an $R$-basis of $\ff_n$ for all $n$. We denote by $c_{b,b'}^{b''}$ the structure constants of $A$ with respect to $\mathcal{B}$, defined by 
$$bb'=\sum_{b''}c_{b,b'}^{b''}b''$$ 
for $b$, $b'$ and $b''\in \mathcal{B}$. 

\begin{lemma}\label{Lem:compatiblebasis}
The basis $\mathcal{B}$ is compatible with $\ff_\bullet$ if and only if there exists a function $w:\mathcal{B} \ra\mathbb{N}$ such that:
\begin{enumerate}
\item $\ff_n$ is spanned by all $b\in\mathcal{B}$ such that $w(b)\leq n$;
\item $w(b'') \leq w(b) + w(b')$ whenever $c_{b,b'}^{b''}\neq 0$.
\end{enumerate}
If this is the case, the classes $\overline{b}$ of elements $b \in \mathcal{B}$ form an $R$-basis of the associated graded algebra $\gr_{\ff_\bullet}A$ with respect to the filtration $\ff_\bullet$, and multiplication in $\gr_{\ff_\bullet}A$ is given by
$$\overline{b}\cdot\overline{b'} = \sum_{b''\in \mathcal{B},\ w(b'')=w(b)+w(b')} \overline{b''}.$$
\end{lemma}

\begin{proof}
If $\mathcal{B}$ is compatible with $\ff_\bullet$, we define $w(b)$ as the minimal $t$ such that $b\in \ff_t$; this function obviously fulfills the above two conditions. All other statements follow from the definitions.
\end{proof}

\begin{definition} 
A function $w$ on isomorphism classes of representations of $Q$ is called
\begin{enumerate}
\item normalized if $w(M) = 0$ only for $M = 0$,
\item weakly admissible if $w(X) \leq w(M) + w(N)$ whenever there exists a short exact sequence $0 \ra N \ra X \ra M \ra 0$;
\item admissible if, additionally, $w$ is additive, that is, we have $w(M \oplus N) = w(M) + w(N)$ for all $M$ and $N$;
\item strongly admissible if it is normalized, admissible, and $w(X) < w(M) + w(N)$ whenever the above exact sequence is non-split.
\end{enumerate}
\end{definition}

We have the following

\begin{corollary}\label{Cor:Admiss}
\begin{enumerate}
\item If $w$ is a weakly admissible function, there exists a filtration $\ff_\bullet$ on $H(Q)$ where $\ff_n$ is spanned by the $F_{[M]}$ such that $w(M) \leq  n$, which is normalized if $w$ is so. 
\item If $w$ is strongly admissible, the associated graded algebra of $H(Q)$ with respect to this filtration is isomorphic to the $q$-commutative polynomial algebra $S_q(\n^-)$.
\end{enumerate}
\end{corollary}

\begin{proof} We apply Lemma \ref{Lem:compatiblebasis} to the basis consisting of the elements $u_{[M]}$ and the function $w(u_{[M]})=w(M)$. Since, by definition of the Hall algebra, a basis element $u_{[X]}$ can only appear with non-zero coefficient $F_{M,N}^X(q)$ in the expansion of $u_{[M]}\cdot u_{[N]}$ provided there exists a short exact sequence as above, the lemma applies.
\end{proof}

Our principal example of a weight function is the following

\begin{definition}\label{deg-fct}
Let $V_0 = \bigoplus_{\alpha\in\Delta_+}U_{\alpha}$ be the direct sum of one copy of each indecomposable. Define $\mu_0(M) = \dim \Hom(V_0, M)$.
\end{definition}

\begin{proposition}\label{Prop:mu0}
The function $\mu_0$ is strongly admissible.
\end{proposition} 

\begin{proof}
Additivity of $\mu_0$ is trivial. An exact sequence $0 \ra M \ra X \ra N \ra 0$ induces an exact sequence
$0 \ra \Hom(V_0, M) \ra \Hom(V_0, X) \ra \Hom(V_0, N)$.
The inequality $\mu_0(X) \leq \mu_0(M)+\mu_0(N)$ follows. If, additionally, equality holds, then $\dim\Hom(V,X)=\dim\Hom(V,M\oplus N)$ for all indecomposables $V$, thus $X \cong M\oplus N$, that is, the above sequence is split. If $\Hom(V_0,M) = 0$, then $M = 0$ since $V_0$ contains some direct summand of $M$.
\end{proof}

Moreover, we can give a complete classification of all (strongly) admissible weight functions.

\begin{theorem}\label{thm:weightfunctions}
\begin{enumerate}
\item A function $w$ is admissible if and only if it is of the form $w(M) = \dim\Hom(V,M)$ for a representation $V$ of $Q$. 
\item Moreover, it is strongly admissible if and only if $V$ contains each non-projective indecomposable representation, and each simple projective representation, as a direct summand.
\end{enumerate}
\end{theorem}

\begin{proof} For an arbitrary representation $V$, we define an additive function $w_V$ by $w_V(M)=\dim{\rm Hom}(V,M)$.

We consider the category $\mathcal{F}$ of $\mk$-linear additive covariant functors from ${\rm rep}_{\mk}Q$ to the category of finite-dimensional $\mk$-vector spaces. Its projective objects are the functors ${\rm Hom}(V,-)$ for $V$ in ${\rm rep}_{\mk}Q$, and its simple objects $S_U$, indexed by the isomorphism classes of indecomposable representations $U$ in ${\rm rep}_{\mk}Q$, are given by $S_U(M)\simeq \mk^l$ if $U$ occurs with multiplicity $l$ as a direct summand of $M$.

For every indecomposable representation $U$, there exists a short exact sequence $0\rightarrow U\rightarrow B\rightarrow \tau^{-1}U\rightarrow 0$ (the Auslander-Reiten sequence starting in $U$) inducing a projective resolution
$$0\rightarrow{\rm Hom}(\tau^{-1}U,-)\rightarrow{\rm Hom}(B,-)\rightarrow{\rm Hom}(U,-)\rightarrow S_U\rightarrow 0$$
of the functor $S_U$. Denoting by $\delta_U$ the additive function with value $1$ on $U$ and with value $0$ on every indecomposable representation not isomorphic to $U$, we see that $\delta_U=w_U+w_{\tau^{-1}U}-w_B$. Thus every additive function  can be written in the form $\sum_Ua_Uw_U$, the sum running over all isomorphism classes of indecomposable representations, for certain $a_U\in{\bf Z}$.

Dually, we consider the category $\mathcal{G}$ of $\mk$-linear additive contravariant functors from ${\rm rep}_\mk Q$ to the category of finite-dimensional $\mk$-vector spaces. Its projective objects are the functors ${\rm Hom}(-,V)$ for $V$ in ${\rm rep}_\mk Q$, and its simple objects $S_U$, indexed by the isomorphism classes of indecomposable representations $U$ in ${\rm rep}_\mk Q$, are again given by $S_U(M)\simeq \mk^l$ if $U$ occurs with multiplicity $l$ as a direct summand of $M$. For every indecomposable representation $U$, there exists a short exact sequence $0\rightarrow \tau U\rightarrow C\rightarrow U\rightarrow 0$ (the Auslander-Reiten sequence ending in $U$) inducing a projective resolution
$$0\rightarrow{\rm Hom}(-, \tau U)\rightarrow{\rm Hom}(-,C)\rightarrow{\rm Hom}(-,U)\rightarrow S_U\rightarrow 0$$
of the functor $S_U$. Applying this sequence of functors to an indecomposable representation $V$, we see that $w_V(\tau U)+w_V(U)-w_V(C)=\delta_{U,V}$. It follows that an additive function $w=\sum_Va_Vw_V$ is admissible if and only if $a_V\geq 0$ for all $V$.

The Auslander-Reiten sequence ending in $U$ is non-split if and only if $U$ is a non-projective indecomposable, thus strong admissibility requires $a_V\geq 1$ for all non-projective indecomposables $V$. Normalization requires that $a_S\geq 1$ for all simple projective representations $S$, since, for $V$ indecomposable ${\rm Hom}(V,S)\not=0$ only if $V$ is isomorphic to $S$. Conversely, if $a_V\geq 1$ for all non-projective indecomposables and all simple projectives $V$, then $w$ is normalized: if $w(M)=0$, then $M$ is projective, and ${\rm Hom}(S,M)=0$ for all simple projectives $S$, thus $M=0$. The above considerations then show that $w$ is strongly admissible.
\end{proof}

\subsection{Remarks}

Since all dimensions of Hom-spaces between indecomposables can, in principle, be read off from a reduced decomposition of the longest Weyl group element which is adapted to the quiver, it is possible to give a purely root-theoretic description of strongly admissible functions. We will not use such a description here since it is less conceptual, but it should be noted that such a description allows for generalization to the non-simply laced cases, which are more difficult to model using Hall algebras.

It should also be noted that the filtrations induced by strongly admissible functions are automatically compatible with the dual canonical basis by the results of \cite{ReDCB}.

\subsection{Example}\label{SSec:Example}
Fix an integer $n\geq 1$, let $Q$ be the quiver of type $A_n$ with equi-orientation as follows:
\[\xymatrix{
1\ar[r] & 2 \ar[r] & \ar[r]\cdots & n-1 \ar[r]& n.}\]
The indecomposable representations of $Q$ are parametrized by $\Delta_+=\{\alpha_i+\cdots+\alpha_j|\ 1\leq i\leq j\leq n\}$, the positive roots in the root system of type $A_n$. For a fixed positive root $\alpha_i+\cdots+\alpha_j\in\Delta_+$, we let $M_{i,j}$ denote the corresponding indecomposable representation. 
\par
We compute the degree of $M_{i,j}$. The proof of the following lemma is a standard exercise.

\begin{lemma}
We have
$$\dim\Hom(M_{r,s},M_{i,j})=\left\{\begin{matrix} 1,& 1\leq i\leq r\leq j\leq s\leq n;\\
0, & \text{otherwise}.\end{matrix}\right.$$
\end{lemma}

\begin{corollary}\label{Cor:Deg}
For $1\leq i\leq j\leq n$, $\deg(M_{i,j})=(j-i+1)(n-j+1)$.
\end{corollary}

\begin{proof}
By definition, the degree of $M_{i,j}$ is 
$$\sum_{1\leq r\leq s\leq n}\dim\Hom\left(M_{r,s},M_{i,j}\right)=\sum_{r=i}^j\sum_{s=j}^n
\dim\Hom(M_{r,s},M_{i,j})=$$
$$=(j-i+1)(n-j+1).$$
\end{proof}

\begin{example}
Let $n=3$. The positive roots are
$$\Delta_+=\{\alpha_1,\alpha_2,\alpha_3,\alpha_1+\alpha_2,\alpha_2+\alpha_3,\alpha_1+\alpha_2+\alpha_3\}.$$
The degrees of the $M_{i,j}$ are given by:
$$\deg(M_{1,1})=3,\ \ \deg(M_{1,2})=4,\ \ \deg(M_{1,3})=3,$$
$$\deg(M_{2,2})=2,\ \ \deg(M_{2,3})=2,\ \ \deg(M_{3,3})=1.$$
\end{example}

\section{Quantum PBW-type filtration \texorpdfstring{in type $A_n$}{for the special linear algebra}}\label{Sec:QGPBW}\label{four}

\subsection{PBW filtration on the quantum group \texorpdfstring{in type $A_n$}{for the special linear algebra}}
Let $\g=\mathfrak{sl}_{n+1}$ be the simple Lie algebra of type $A_n$ and $U_q(\g)$ be the associated quantum group.
\par
Let $Q$ be the Dynkin quiver of type $A_n$ with orientation as in Section \ref{SSec:Example}. Let $H(Q)$ be the Hall algebra of $Q$. Recall that there exists an isomorphism of algebras $\eta:U_q(\mathfrak{n}^-)\ra H(Q)$. The grading $\mu_0$ defined on $H(Q)$ can be pulled back to $U_q(\mathfrak{n}^-)$ via this isomorphism. We let $\ff_n U_q$ denote the set of elements in $U_q(\mathfrak{n}^-)$ of degree no more than $n$.
\par
We fix a reduced decomposition of $w_0$ as in Section \ref{sec-2-1} and Section~\ref{Sec:Hall} to construct the quantum PBW root vectors and the degree function. For $1\leq i\leq j\leq N$, we let $F_{i,j}$ denote the quantum PBW root vector associated to the positive root $\alpha_i+\cdots+\alpha_j$. 
\par
Recall the specialization map $\xi:U_q(\mathfrak{n}^-)\ra U(\mathfrak{n}^-)$. Without loss of generality, we may assume that a basis $\{f_{i,j}|\ 1\leq i\leq j\leq n\}$ of $\mathfrak{n}^-$ is chosen such that $\xi(F_{i,j})=f_{i,j}$. By Corollary \ref{Cor:Deg}, 
$$\deg(F_{i,j})=(j-i+1)(n-j+1).$$
\par
Combining Lemma \ref{Lem:LSHall}, Corollary \ref{Cor:Admiss} and Proposition \ref{Prop:mu0}, we obtain:

\begin{proposition}\label{Prop:qFilt}
\begin{enumerate}
\item The filtration 
$$\ff_\bullet U_q=(\ff_0 U_q\subset \ff_1 U_q\subset\cdots\subset \ff_n U_q\subset\cdots)$$ 
endows $U_q(\mathfrak{n}^-)$ with a filtered algebra structure.
\item The associated graded algebra ${\gr}_{\ff}U_q$ is a $q$-commutative polynomial algebra isomorphic to $S_q(\mathfrak{n}^-)$.
\end{enumerate}
\end{proposition}

As promised, we give a proof of Proposition \ref{Prop:Filt}.

\begin{proof}[Proof of Proposition \ref{Prop:Filt}]
By Proposition \ref{Prop:SpePBW}, Proposition \ref{Prop:Filt} is a specialization of Proposition \ref{Prop:qFilt}.
\end{proof}

\subsection{Main theorem}
The filtration $\ff_\bullet$ on $U_q(\mathfrak{n}^-)$ defines a filtration on $V_q(\lambda)$ by letting $\ff_sV_q(\lambda)=\ff_sU_q.\mathbf{v}_\lambda$. Let $V_q^\ff(\lambda)$ denote the associated graded vector space.
\par
The $U_q(\mathfrak{n}^-)$-module structure on $V_q(\lambda)$ induces a cyclic $S_q(\mathfrak{n}^-)$-module structure on $V_q^\ff(\lambda)$, i.e., $V_q^\ff(\lambda)=S_q(\mathfrak{n}^-).\mathbf{v}_\lambda$. Let $\psi:S_q(\mathfrak{n}^-)\ra V_q^\ff(\lambda)$ be the $S_q(\mathfrak{n}^-)$-module morphism defined by $x\mapsto x.\mathbf{v}_\lambda$. It is clearly surjective. Let $I_q^\ff(\lambda)$ denote the kernel of $\psi$, then as $S_q(\mathfrak{n}^-)$-modules, $V_q^\ff(\lambda)\cong S_q(\mathfrak{n}^-)/I_q^\ff(\lambda)$.

\begin{theorem}
The set $\{F^{\mathbf{p}}\mathbf{v}_\lambda|\ \mathbf{p}\in S(\lambda)\}$ forms a basis of $V_q^\ff(\lambda)$ and the annihilating ideal $I_q^\ff(\lambda)$ is monomial.
\end{theorem}

\begin{proof}
Since $\dim V_q^\ff(\lambda)=\dim V_q(\lambda)=\dim V(\lambda)=\# S(\lambda)$ (the last equality arises from Theorem~\ref{ffl2}), it suffices to show that the set $\{F^{\mathbf{p}}\mathbf{v}_\lambda|\ \mathbf{p}\in S(\lambda)\}$ is linearly independent in $V_q^\ff(\lambda)$.
\par
Suppose that there exists a linear combination
$$\sum_{\mathbf{p}\in S(\lambda)}c_{\mathbf{p}}F^{\mathbf{p}}\mathbf{v}_\lambda=0$$
for $c_{\mathbf{p}}\in\mathbf{A}_1$. We separate these coefficients into two groups:
$$S(\lambda)_0=\{{\mathbf{p}}\in S(\lambda)|\ {\ev}_1(c_{\mathbf{p}})=0\},\ \ S(\lambda)_1=\{{\mathbf{p}}\in S(\lambda)|\ {\ev}_1(c_{\mathbf{p}})\neq 0\},$$
and we write the summation into two groups:
$$\sum_{\mathbf{p}\in S(\lambda)_0}c_{\mathbf{p}}F^{\mathbf{p}}\mathbf{v}_\lambda+\sum_{\mathbf{p}\in S(\lambda)_1}c_{\mathbf{p}}F^{\mathbf{p}}\mathbf{v}_\lambda=0.$$
By specializing $q$ to $1$, the first summand gives $0$ and the second one gives 
$$\sum_{\mathbf{p}\in S(\lambda)_1}{\ev}_1(c_{\mathbf{p}})f^{\mathbf{p}}v_\lambda=0$$
with ${\ev}_1(c_{\mathbf{p}})\neq 0$. 
\par
By Theorem \ref{main-thm}, $\ev_1(c_{\mathbf{p}})=0$, in contradiction to the definition of $S(\lambda)_1$, and forcing $S(\lambda)_1=\emptyset$. That is to say, for any ${\mathbf{p}}\in S(\lambda)$, $c_{\mathbf{p}}\in (q-1)\mathbf{A}_1$. Let $m$ be the maximal integer such that for any ${\mathbf{p}}\in S(\lambda)$, $c_{\mathbf{p}}\in (q-1)^m\mathbf{A}_1$. Hence for any ${\mathbf{p}}\in S(\lambda)$, 
$$\frac{c_{\mathbf{p}}}{(q-1)^m}\in\mathbf{A}_1.$$
We consider a new linear combination by dividing $(q-1)^m$:
$$\sum_{\mathbf{p}\in S(\lambda)}\frac{c_{\mathbf{p}}}{(q-1)^m}F^{\mathbf{p}}\mathbf{v}_\lambda=0.$$
In this new linear combination, $S(\lambda)_1\neq \emptyset$, but repeating the above procedure gives the contradiction $S(\lambda)_1=\emptyset$. Hence $c_{\mathbf{p}}=0$ for any $\mathbf{p}\in S(\lambda)$. 
\par
Since the specialization map $\xi$ preserves the filtration, the monomiality of the annihilating ideal follows from Theorem \ref{main-thm}.
\end{proof}

\section*{Acknowledgments}
The work of Xin Fang is supported by the Alexander von Humboldt Foundation. The work of Ghislain Fourier is funded by the DFG priority program 1388 ''Representation Theory''.

\end{document}